\documentclass[11pt]{amsart}
\usepackage{a4,tikz,color}
\usetikzlibrary{arrows}

\newcommand\N{\mathbb{N}}
\newcommand\R{\mathbb{R}}
\newcommand\sset{\mathbb{S}}
\newcommand\Z{\mathbb{Z}}

\newcommand\p{\varphi}
\newcommand{\Wp}{{W_\p}}

\newcommand\distf{{\rm dist}_\p}

\newcommand{\HH}{{\mathcal H}}
\def\Om{{\Omega}}  
\def\div{\mathop{\mathrm{div}}}

\def\eps{\varepsilon}
\def\j{{j_{\rm max}}}
\def\1{{\bf 1}}

\newtheorem{theorem}{Theorem}[section]
\newtheorem{definition}[theorem]{Definition}
\newtheorem{lemma}[theorem]{Lemma}
\newtheorem{proposition}[theorem]{Proposition}
\newtheorem{corollary}[theorem]{Corollary}
\theoremstyle{remark}
\newtheorem{remark}[theorem]{Remark}

\numberwithin{equation}{section}

\title{Droplet condensation and isoperimetric towers}

\author{Matteo Novaga}
\address[Matteo Novaga]{Dipartimento di Matematica, Universit\`a di Padova, via Trieste 63, 35121 Padova, Italy}
\email{novaga@math.unipd.it}
\author{Andrei Sobolevski}
\address[Andrei Sobolevski]{Institute for Information Transmission Problems (Kharkevich
  Institute), 19 B. Karetny per., 127994 Moscow, Russia
  \and
National Research University Higher School of Economics, 20
Myasnitskaya st., Moscow 101000, Russia}
\email{sobolevski@iitp.ru}
\author{Eugene Stepanov}%
\address[Eugene Stepanov]{
St.Petersburg Branch
of the Steklov Mathematical Institute of the Russian Academy of Sciences,
Fontanka 27,
191023 St.Petersburg,
Russia
\and
Department of Mathematical Physics, Faculty of Mathematics and Mechanics,
St. Petersburg State University, Universitet\-skij pr.~28, Old Peterhof,
198504 St.Peters\-burg, Russia}
\email{stepanov.eugene@gmail.com}
\thanks{
The first author acknowledges partial support by the Fondazione CaRiPaRo Project
``Nonlinear Partial Differential Equations: models, analysis, and
control-theoretic problems.''
The second author was partially supported by Laboratory for Structural
Methods of Data Analysis in Predictive Modeling, MIPT, RF government
grant, ag. 11.G34.31.0073, and by RFBR grant 11-01-93106 CNRSL\_a.
The work of the third author was financed by GNAMPA, by RFBR grant 11-01-00825,
and by the project 2008K7Z249 ``Trasporto ottimo di massa,
disuguaglianze geometriche e funzionali e applicazioni'' financed by the
Italian Ministry of Research.
}
\date{}

\begin{document}

\begin{abstract}
  We consider a variational problem in a planar convex domain, motivated by statistical mechanics
  of crystal growth in a saturated solution. The minimizers are constructed explicitly and are completely characterized.
\end{abstract}

\maketitle

\section{Introduction}\label{sec_mcheeg_intro}

In understanding the physical phenomenon of droplet condensation or
crystal growth, the central issue is to explain how a particular
macroscopic shape of the growing droplet or crystal is determined by
microscopic interactions of its constituent paricles.

According to Gibbs' formulation of statistical mechanics, the
probability of a microscopic configuration~$\sigma$ is proportional to
$\exp(-\beta H(\sigma))$, where $\beta > 0$ is the inverse temperature
and $H(\cdot)$ is the Hamiltonian defining the energy of the system.
Therefore the most probable configurations are the ones with minimal
energy.  In the ``thermodynamical'' limit of a large number of
particles, this minimum becomes very sharp: the overall configuration
of the system settles, up to minute fluctuations, to a well-defined
deterministic structure.

It turns out that the microscopic laws of atomic
interactions give rise to a certain macroscopic quantity, the surface
tension, which determines the droplet shape via minimization of the
surface energy.  Phenomenology of surface tension has been proposed by
Gibbs in the late 1870's.  In an important contribution, G.~Wulff
suggested in 1900 that for a growing crystal, its equilibrium shape is
that of a ball in a metric generated by the surface tension (the
\emph{Wulff shape}).

It has been furthermore observed experimentally that flat facets of a
growing crystal may carry macroscopic but monomolecular ``islands'',
whose shape is also determined by the surface tension.  A mathematical
approach to explanation of this phenomenon has been developed by
S.~Shlosman and collaborators in a series of works \cite{Schonmann.R:1996,IoffeShlos08,IoffeShlos10},
building upon his earlier work with R.~L.~Dobrushin and R.~Kotecky \cite{DobrKotShlos92}.

A typical setting in this approach is represented by the following discrete model
of crystal growth, which is a variant of the Ising model: fix an open
domain $\Omega\subset \R^2$ of unit area and consider the
three-dimensional lattice obtained by intersecting the cylinder
$\Omega \times [-1, 1] \subset \R^3$ with $\frac 1N (\Z^3 + (0, 0,
\frac 12))$, where $N$ is a large integer parameter.  At each node~$t$
of this lattice there is a variable~$\sigma_t$ (the \emph{spin})
taking values $+1$ (interpreted as ``$t$ belongs to the free phase'')
and $-1$ (interpreted as ``$t$ belongs to the condensed phase'').  The
collection $\sigma = (\sigma_t)$ is called the microscopic
configuration of the system.

Fix now the Ising Hamiltonian $H(\sigma) = - \sum_{s, t\colon |s -
  t| = 1} \sigma_s \sigma_t$, which describes a ``ferromagnetic''
interaction between nearest neighbors (equal values have smaller
energy than opposite ones), and consider the canonical probability
distribution $p(\sigma) = \exp(-\beta H(\sigma))/Z$.  Here the
normalization coefficient $Z = \sum_\sigma \exp(-\beta H(\sigma))$ is
defined by summation over all configurations that satisfy the
so-called \emph{Dobrushin boundary condition}: spins at outermost nodes $(x, y,
z)$ of the lattice have values $+1$ if $z > 0$ and $-1$ if $z < 0$.

It turns out that in the limit of large~$N$ the main contribution to
probability comes from configurations where the lower and upper halves
of the lattice are filled respectively with $-1$'s and $+1$'s.  In
this equilibrium state, the numbers of $+1$'s and $-1$'s are
asymptotically equal, so that $S_N = \sum_t \sigma_t \sim 0$, and
fluctuations of the flat surface dividing the two phases are
logarithmic in~$N$.

A more interesting situation occurs when, in addition to the Dobrushin
boundary values, the system is conditioned to have macroscopically
more $-1$'s than $+1$'s:
\begin{displaymath}
  S_N = \sum_t \sigma_t = -m N^2
\end{displaymath}
with $m > 0$.  In this case, depending on the value of~$m$, the most
probable state of the system may feature one or more monomolecular
layers on top of the surface $z = 0$ in the box $\Om \times [-1, 1]$.
A detailed account of the observed equilibrium states as $m$ changes
can be found in~\cite{IoffeShlos10}.

As proved in~\cite{Schonmann.R:1996},
the behavior of this model
in the continuous limit $N \to \infty$ is closely related to the
following variational problem: given an open set $\Om\subset\R^n$ and a
value $m\in [0,+\infty)$, find
\begin{equation}\label{varpb}
  \min \left\{ \int_\Om\p^*(Du)\colon
    u\in BV(\R^n),\ \text{$u = 0$ on~$\R^n\setminus\Om$},\ u(\cdot)\in \N,\
    \int_\Om u\,dx = m \right\},
\end{equation}
where $\p^*$ is some given general norm on $\R^n$.  Of course, in the application to the Ising model
we are discussing here one has $n=2$, i.e.\ one works in the two-dimensional case; however the case of generic dimension
$n$ of the ambient space $\R^n$ also makes sense from the mathematical point of view.
The growth of a
droplet and formation of new layers of the solid is described by the
growth of profile $u$ as $m$ increases.

The norm $\p^*(\cdot)$ here is related to the surface tension as
follows.  The surface tension $\gamma^{3D}(\cdot)$ is a function defined
over $\sset^2$, the two-dimensional unit sphere in~$\R^3$, and
satisfying $\gamma^{3D}(\nu) \ge 0$ and $\gamma^{3D}(-\nu) = \gamma^{3D}(\nu)$ for
all $\nu \in \sset^2$.  The surface energy of a closed surface $M^2
\subset \R^3$ is defined to be
\begin{displaymath}
  H(M^2) = \int_{M^2} \gamma^{3D}(\nu_s)\, ds,
\end{displaymath}
where $\nu_s$ is the unit normal to~$M^2$ at~$s\in M^2$.  While $\gamma^{3D}$
defines the $3D$ shape of a crystal growing in space, the shape of
monolayers growing on facets is given by the restricted $2D$ surface
tension defined for $n \in \sset^1$ by
\begin{displaymath}
  \gamma^{2D}(\nu) = \frac \partial {\partial \nu} \gamma^{3D}\Bigr|_{\nu_s = (0, 0, 1)},
\end{displaymath}
where the derivatives are taken at the ``north pole'' $\nu_s = (0, 0, 1)
\in \sset^2$ along all tangents $\nu \in \sset^1$ to~$\sset^2$~\cite{IoffeShlos10}.
The function $\gamma^{2D}$ can then be extended to
the whole $\R^2$ by homogeneity of degree one, and $\p^*(\cdot)$ is
defined as the convex hull of thus defined~$\gamma^{2D}(\cdot)$.
However in the sequel $\p^*$ will be fixed, without any
assumptions of smoothness or strict convexity
(indeed one of the examples in Section~\ref{secexa} corresponds to a crystalline norm).

It is easy to see that the functional minimized in~\eqref{varpb} is
the one-dimensional surface energy for the restricted surface tension.
It turns out that minimization of this surface energy alone is
sufficient to reconstruct most of the physics of monomolecular layers
growth described in~\cite{IoffeShlos10}.
In particular, if $\p^*(\cdot)$ is the Euclidean norm and $\Om$ a
unit square, then as $m$ grows, the
first four mono\-mo\-lecular layers start as Wulff circles and then
develop into ``Wulff plaquettes'' while from the fifth
layer on all new layers appear as Wulff plaquettes identical to
the underlying layers (Section~\ref{secexa}).

In contrast, this simple variational model does not capture the
thermodynamic fluctuations, which render Wulff circles below a certain
size unstable and prevent their formation for small~$m$.  Neither does
it capture the microscopic (i.e., ``finite-$N$'') structure of the
Wulff plaquettes, whose boundaries are in fact separated with gaps
that vanish in the continuous limit.  A first-principle approach that
takes proper account of these phenomena is due to R.~Dobrushin,
S.~Shlosman and their coauthors and is presented in their works
\cite{DobrKotShlos92,Schonmann.R:1996,IoffeShlos96,IoffeShlos08,IoffeShlos10,IoffeShlos12}.

It is worth observing that a similar problem with additional restriction that $u$ be a characteristic function
of some set (i.e.\ that the droplet has exactly one layer)
in the two-dimensional situation (i.e.\ when $n=2$), the set $\Omega$ is convex, and the norm $\p^*$ is Euclidean,
has been studied in~\cite{StredulZiem97}, and for more general anisotropic
norms (but for a somewhat different functional, namely, with penalization on the volume instead of the volume
constraint), in~\cite{NovPao05}. The latter problem will play an important role also in the present paper.
Eventually, one has to mention that it is also very similar to the well-known Cheeger problem,
the solutions of the latter being so-called Cheeger sets (see e.g.~\cite{ButCarComt07,KawNov08,KawLach06,CasShamNov10}).

Our aim in this paper is to study the variational problem
\eqref{varpb} in the two-dimensional case (i.e.\ when $n=2$). %
This geometric optimization problem is considered without resort to
the underlying lattice model or its continuous limit, allowing us to
treat an arbitrary open domain $\Om$ and an arbitrary norm $\p^*$ that
is not necessarily strictly convex. %
In this setting we completely characterize the minimizers and the
possible levels of~$u$ when the domain $\Om$ is convex. %
In particular it turns out that that except some degenerate situation,
which can however happen only when $\Omega$ is not strictly convex,
the number of nonzero levels of $u$ is at most two. %

The basic tool we use is the auxiliary problem when $u$ is a priori
required to have a single nonzero level (i.e.\ is requested to be a
characteristic function); namely, we show that in the two-dimensional
case ($n=2$) when $\Om$ is convex, the nonzero levels of solutions to the latter problem corresponding to different values of $m$ as $m$ grows
can be arranged as a family of sets ordered by inclusion. Thus, solutions to problem~\eqref{varpb}
can be seen as ``towers'' with levels solving the auxiliary problem.
The assumption of convexity of~$\Om$ is essential, as shown by a
counterexample at the end of Subsection~\ref{sec:isoperimetric-sets}.
The main result of the paper is formulated as Theorem~\ref{th_mcheeg1final}.
We conclude with an explicit example of solutions to~\eqref{varpb} for the case of a square $\Omega=[0,1]^2$ with
the Euclidean norm and a crystalline norm.

This work was inspired by some seminar talks of Senya Shlosman.  After
it was completed, we learned that a full description of the solutions
to the variational problem~\eqref{varpb} when $\Omega$ is a square and
$\p^*$ is generated by a physical Hamiltonian (in particular, is the
Euclidean norm) has been independently obtained by him and
Ioffe~\cite{shlosman} by a rigorous continuous limit of a suitable
lattice model.  Their proof, together with an analysis of the
microscopic structure of the solution and its behavior under thermal
perturbations, will appear in the forthcoming
publication~\cite{IoffeShlos12}.

\section{Notation and preliminary results}\label{sec_mcheeg_not}

For a set $E\subset \R^n$ we denote by $|E|$ its Lebesgue measure, by $\1_E$ its characteristic function,
by $\bar E$ its closure, by $\partial E$ its topological boundary, and by $E^c$ its complement.

In the following $\p$ will denote the given (not necessarily Euclidean) norm over $\R^n$.
Given $E\subset \R^n$ and $x \in \R^n$, we set
\begin{equation*}
\distf(x,E) := \inf_{y\in E}\p(x-y), \quad
d^E_\p(x) := \distf(x,E) - \distf(x,E^c).
\end{equation*}
The value $d^E_\p(x)$ is the signed distance from $x$ to $\partial E$ and is positive outside $E$. Notice that at each point where $d_\p^E$ is differentiable one has (see~\cite{BelNovMPao01-II})
\begin{equation}\label{unno}
\p^*(\nabla d^E_\p) = 1, \quad \qquad
\nu\cdot \nabla d^E_\p = 1 \quad \mbox{ for all } \nu\in\partial \p^*(\nabla d^E_\p),
\end{equation}
where $\p^*$ denotes the dual norm of $\p$ defined as
\[
\p^*(\xi) := \max\{\xi\cdot\eta\,:\, \p(\eta)\le 1\} 
\]
and $\partial \p^*$ denotes the subdifferential of $\p^*$ in the sense of convex analysis.
In particular
\[
\nabla d^E_\p = \frac{\nu^E}{\p^*(\nu^E)}
\]
where $\nu^E$ is the exterior Euclidean unit normal to $\partial E$.

We define the anisotropic perimeter of  a set $E\subseteq\R^n$ as
\begin{equation}\label{defper}
P_\p(E) := \sup \left\{ \int_E \div\eta \ dx\colon \eta\in C^1_0(\R^n),\,\p(\eta)\le 1\right\}
= \int_{\partial^* E}\p^*(\nu^E) d\mathcal H^{n-1},
\end{equation}
where $C^1_0(\R^n)$ stands for the set of continuously differentiable functions with compact support is $\R^n$,
$\partial^* E$ is the reduced boundary of $E$ according to De Giorgi and $\mathcal H^k$ stands for the $k$-dimensional
Hausdorff measure.
We will usually identify a set $E$ of finite perimeter with the set of its density points (i.e.\ points
of density $1$).

Given an open set $\Om\subset\R^n$ we define
the $BV$-seminorm of $v\in BV(\Om)$ as
\[
\int_\Om\p^*(Dv) :=
\sup \left\{ \int_\Om v \,\div\eta \ dx:\, \eta\in C^1_0(\R^n),\,\p(\eta)\le 1\right\}.
\]

We let $\Wp :=\{ x|~\p(x)< 1\}$, usually called the \emph{Wulff shape}, be the unit ball
of $\p$. Observe that $P_\p(W_\p)=n|W_\p|$.

In the sequel, given $x\in\R^n$ and $r>0$, we set $W_r(x):=
x+rW_\p$ (a \emph{Wulff ball} of radius $r$ with center $x$). In this notation the reference to
a norm $\p$ is not retained for the sake of brevity, but always silently assumed.
When $\p$ is the Euclidean norm, we will
use a more common notation $B_r(x)$ instead of $W_r(x)$ and $P$ instead of $P_\p$.

\begin{definition}
Given an $r>0$, we say that $E$ satisfies the $r\Wp$-condition, if for every $x\in\partial E$ there exists an $y\in\R^n$
such that
\[
W_r(y)\subset E \qquad \text{and}\qquad
x\in \partial W_r(y).
\]
\end{definition}

Observe that, if $E$ is convex, then $E^c$ satisfies the $r\Wp$-condition for all $r>0$.

We conclude the section by recalling the following isoperimetric inequality~\cite{TaylorJ75}.

\begin{proposition}
For all $E\subset\R^n$ such that $|E|<+\infty$ there holds
\begin{equation}\label{isop}
P_\p(E)\ge \frac{|E|^\frac{n-1}{n}}{|W_\p|^\frac{n-1}{n}}\,P_\p(W_\p).
\end{equation}
\end{proposition}

\section{Existence of minimizers}

Notice that, since the total variation is lower semicontinuous and the constraints are closed
under weak $BV$ convergence, by direct method of the calculus of variations one immediately gets
existence of  minimizers of~\eqref{varpb}.

\begin{proposition}
For any $m\ge 0$ there exists a (possibly nonunique) minimizer of~\eqref{varpb}.
\end{proposition}

For every $u\in L^1(\R^n)$ and $j\in \N$ we set
\begin{equation}\label{defEj}
E_j:=\{ u\ge j\}.
\end{equation}
It is worth observing that
whenever $u(\cdot)$ takes values in $\N$, one has
\begin{equation}\label{eq_mcheeg1int1}
u= 
\sum_{i=1}^\infty  \1_{E_i}
\end{equation}
and
\begin{equation}\label{eq_mcheeg1int2}
\int_{\R^n}\p^*(Du)= \sum_{i=1}^\infty  P_\p(E_i).
\end{equation}

\begin{remark}\label{remlim}
It is worth observing that, if we let $u_m$ be a minimizer of~\eqref{varpb} for a given $m>0$, then
the normalized functions $v_m:=u_m/m$ converge, as $m \to \infty$, up to a subsequence, to a minimizer of the problem
\[
\min \left\{ \int_\Om\p^*(Dv)\colon v\in BV(\R^n), \text{$v=0$ on
    $\Omega^c$},\ \int_\Om v\,dx = 1 \right\}
\]
which is closely related to the so-called Cheeger problem in $\Om$~\cite{KawNov08}.
\end{remark}

The following assertions hold true.

\begin{proposition}\label{prop_jmax}
If $u$ is a minimizer of~\eqref{varpb}, then $u\in L^\infty(\R^n)$.
\end{proposition}

\begin{proof}
Assume by contradiction that $|E_j|>0$ for all $j\in\N$.
Notice that \[
\lim_{j\to\infty}|E_j|=0
\]
(since otherwise $u$ would not be integrable).
Given $x_0\in\Om$ we let
\[
u_j := \min(u, j) + \1_{W_{R_j}(x_0)}
\]
where the radius $R_j$ is such that
\[
\int_\Om u_j = \int_\Om u=m,
\]
that is (keeping in mind~\eqref{eq_mcheeg1int1})
\[
|W_\p| R_j^n = \sum_{i>j} |E_i|,
\]
and choose $j\in \N$ big enough so that $W_{R_j}(x_0)\subset\Om$.

Letting $f(t) := n |W_\p|^{1/n}  t^{(n-1)/n}$,
so that
\[
P_\p(W_{R_j}(x_0))=f(|W_\p| R_j^n),
\]
we have
\[
\begin{array}{rll}
\displaystyle\int_\Om\p^*(Du_j) & \le \displaystyle\int_\Om\p^*(D\min(u,j)) + P_\p(W_{R_j}(x_0))\\
&=
\displaystyle\int_\Om\p^*(D\min(u,j)) + f(|W_\p| R_j^n)\\
& \le \displaystyle\int_\Om\p^*(D\min(u,j)) + \sum_{i>j} f(|E_i|) & \mbox{ by the concavity of $f$}\\
&\le \displaystyle\int_\Om\p^*(D\min(u,j)) + \sum_{i>j} P_\p(E_i) & \mbox{ by 
\eqref{isop}} \\
&= \displaystyle\int_\Om\p^*(Du) & \mbox{ by~\eqref{eq_mcheeg1int2}},
\end{array}
\]
the second inequality being strict unless $|E_i|=|E_k|$ for all $i>j$, $k>j$, thus leading to a contradiction.
\end{proof}

\begin{proposition}\label{prop_rel1}
Let $\Om\subset \R^n$ be star-shaped. Then
Problem~\eqref{varpb} is equivalent to the following relaxed problem
\begin{equation}\label{pbrel}
\min \left\{ \int_\Om\p^*(Du):\, u\in BV(\R^n),\ \text{$u=0$ on
    $\Omega^c$},\ u(\cdot)\in \N,\ \int_\Om u\,dx \ge m \right\}.
\end{equation}
Namely, the minimum values and the minimizers are the same for both problems.
\end{proposition}

\begin{proof}
It is enough to show that any minimizer $u$ of~\eqref{pbrel} satisfies
\begin{equation}\label{eqvol}
\int_\Om u\,dx = m.
\end{equation}
To this aim let $\Om$ be star-shaped with respect to $x_0$ and assume by contradiction that \eqref{eqvol} is violated. 
Let $u_\lambda(x):=u(x_0+\lambda(x-x_0))$ for any $\lambda>0$, so that
$u_\lambda\in BV(\R^n)$, $u_\lambda(\cdot)\in \N$, while, by star-shapedness of $\Om$, one has $u_\lambda=0$ outside of $\Om$ for every $\lambda\geq 1$.
Then there exists a $\lambda>1$ such that~\eqref{eqvol} holds with $u$ replaced by $u_\lambda$.
However
\[
\int_\Om\p^*(Du_\lambda) = \lambda^{1-n}\int_{x_0+\lambda(\Om-x_0)}\p^*(Du)
= \lambda^{1-n}\int_\Om \p^*(Du) < \int_\Om \p^*(Du)
\]
(the second equality is due to the fact that $\Om\subset
x_0+\lambda(\Om-x_0)$ for $\lambda >1$, while $u=0$ outside of $\Om$),
contradicting the minimality of $u$.
\end{proof}

\section{The convex two-dimensional case}\label{secconvex}

In this section we shall assume that $n = 2$ and $\Om\subset\R^2$ is a convex open set.

Given $E\subset \R^2$ and an $r>0$, we define the set $E^r\subset E$ by the formula
\begin{equation}\label{defomr}
E^r:=
\left\{
\begin{array}{ll}
\bigcup\left\{ W_r(x):\, W_r(x)\subset E\right\}, & {\rm if\ }r>0,
\\
\\
E, & {\rm if\ }r=0.
\end{array}
\right.
\end{equation}
Notice that, if $E$ is a convex set, then $E^r$ is convex
and satisfies the $rW_\p$-condition. The set $E^r$ is called the~\emph{Wulff plaquette} of radius~$r$
relative to $E$.

The following assertion holds.

\begin{lemma}\label{lm_EeqEr}
Let $E\subset \R^2$ be a convex open set satisfying the $rW_\p$-condition for some $r>0$, then $E=E^r$.
\end{lemma}

\begin{proof}
One has $E^r\subset E$. On the other hand, $\partial E\subset \partial E^r$ because $E$ satisfies the $rW_\p$-condition.
Minding that $E$, and hence $E^r$, is convex, we get $E=E^r$.
\end{proof}

It is worth mentioning that convexity of the set $E$ is essential in the above Lemma~\ref{lm_EeqEr}. In fact, if
$A$, $B$ and $C$ are the vertices of an equilateral triangle $\triangle ABC$ with sidelength $1$, then letting
\[
E:= B_{1/2}(A)\cup B_{1/2}(B)\cup B_{1/2}(C)\cup \triangle ABC
\]
we have that $E$ satisfies the $\frac 1 2 W_\p$-condition with respect to the Euclidean norm, but
\[
E^{1/2} = B_{1/2}(A)\cup B_{1/2}(B)\cup B_{1/2}(C) \neq E.
\]

\subsection{Isoperimetric sets}\label{sec:isoperimetric-sets}

We consider the constrained isoperimetric problem
\begin{equation}\label{isopb}
\min \left\{ P_\p(E):\, E\subset\Om,\, |E| =m\in [0,|\Om|] \right\}
\end{equation}
which corresponds to Problem~\eqref{varpb} under the additional constraint that $u$
is a characteristic function. Clearly, the minimizers to this problem exist and the assertion of
Proposition~\ref{prop_rel1} remains valid for this problem.

Let $R_\Om>0$ be the maximal radius $R$ such that $W_R(x)\subseteq\Om$ for some $x\in\Om$, and let
$r_\Om\in [0,R_\Om]$ be the maximal radius $r$ such that $\Om$ satisfies the $rW_\p$-condition
(we set for convenience $r_\Om:=0$ if $\Om$ does not satisfy any $rW_\p$-condition).
Observe that in the Euclidean case one has
\[
r_\Om= \frac{1}{\|\kappa\|_{L^\infty(\partial \Om)}}
\]
where $\kappa$ stands for the curvature of $\partial\Om$.

\begin{lemma}\label{lm_isoper1}
Let $m\in (0,|\Om|)$,
and let $E$ be a minimizer of~\eqref{isopb}. Then $E$ is convex and
there exists an $r>0$ (depending on $m$) such that $E$ satisfies the $rW_\p$-condition and each connected component
of $\partial E\cap \Om$ is contained in $\partial W_r(x)$, for some
Wulff ball $W_r(x)\subset \Om$ (with $x$ depending on the connected component of $\partial E\cap \Om$).
\end{lemma}

\begin{remark}
Recall that here and in the sequel when speaking of the properties of a set $E$ of finite perimeter
we actually refer to the respective properties of the set of its density points. In particular, a
minimizer $E$ of~\eqref{isopb} is not necessarily convex, but the set of its density points is (and hence, in particular, the closure $\bar E$ is convex).
\end{remark}

\begin{proof}
We divide the proof into four steps.

{\sc Step 1.} Let us first show the convexity of $E$. As in~\cite[Theorem 2]{ACMM}
we can uniquely decompose $E$ as a union of (measure theoretic) connected components
$\{E_i\}_{i\in I}$, where $I$ is finite or countable, such that
\[
|E|=\sum_{i\in I}|E_i|\qquad {\rm and}\qquad P_\p(E)=\sum_{i\in I}P_\p(E_i).
\]

As in~\cite[Proposition~6.12]{AmbNovPao02}, one shows
by the isoperimetric inequality and the minimality of $E$,
that the number of connected components is finite
and the boundary of each connected component $E_i$ is parameterized
by a finite number of pairwise disjoint Jordan curves. In particular,
the boundaries of two different connected components do not intersect.
Further, using lemma~6.9 from~\cite{AmbNovPao02} one has that
the perimeter $P_\p(E_i)$ of a measure theoretic connected component $E_i$
having the boundary parameterized by Jordan curves
$\{\theta_i^j\}_{j=1}^{N_i}$ (all parameterized, say, over $[0,1]$) is given by
\[
P_\p (E_i)= \sum_{j=1}^{N_i} \int_0^1 \psi(\dot{\theta}_i^j(t))\,dt,
\]
where
$\psi \colon \R^2\to \R$ is some convex and $1$-homogeneous functions
(in fact, $\psi := \p^*\circ R$, $R$ being the clockwise rotation of $\R^2$ by $\pi/2$, see
corollary~6.10 from~\cite{AmbNovPao02}).
Hence, using Jensen inequality one shows that
the convex envelope of $E_i$ has lower (anisotropic) perimeter than $E_i$ itself, and minding that it also has greater volume
(as well as the fact that the assertion of
Proposition~\ref{prop_rel1} is valid for Problem~\eqref{isopb}), one has that each $E_i$ is convex.

Finally, if $E$ is not connected,
recalling that $\Omega$ is convex we can translate a connected component inside $\Omega$ in such a way that
its boundary touches the boundary of another connected component (this does not change neither the perimeter nor the volume),
and taking the convex envelope of the resulting set
we obtain again a set with greater volume and strictly lower
anisotropic perimeter, hence a contradiction which shows that
$E$ is convex.

{\sc Step 2.} Reasoning as in~\cite[theorem~4.5]{NovPao05},
where the authors consider the related problem
$\min \left\{ P_\p(E)-\lambda|E|:\, E\subset\Om,\, \lambda\ge 0 \right\}$
instead of \eqref{isopb}, one gets that each connected component
of $\partial E\cap \Om$ is contained in $\partial W_{r}(x)$, for some $x\in\R^2$ and $r>0$.

Moreover, as in~\cite[theorem~6.19]{AmbNovPao02} one can show the existence of a (possibly nonunique) Lipschitz continuous vector field
$n\colon\partial E\to \R^2$ such that $n(x)\in\partial \varphi^*(\nu(x))$ for $\HH^1$-a.e. $x\in \partial E$.
In particular ${\div}_\tau n\in L^\infty(\partial E)$, where ${\div}_\tau n := \partial_\tau(n\cdot\tau)$ denotes the tangential divergence
of $n$ (here and below $\tau$ and $\nu$ denote the Euclidean unit tangent and exterior normal vectors to $\partial E$ respectively) and corresponds to the anisotropic curvature of $\partial E$
(cfr.~\cite{TaylorJ75,BelNovMPao01-II}).

Without loss of generality we may assume that ${\div}_\tau n$ is constant along every maximal segment contained in $\partial E$
(if not, we can substitute $n$ over the segment by a convex combination of its values on the endpoints of the segment;
one would then still have $n\in \partial \varphi^*(\nu)$ along the segment because $\nu$ is constant there and
$\partial \varphi^*(\cdot)$ is convex). In particular, if a connected component $\Sigma$
of $\partial E\cap \Om$ is contained in $\partial W_{r}(x)$,
then $n(y)=(y-x)/(r \varphi(y-x))$
for $\HH^1$-a.e. $y\in\Sigma$.

{\sc Step 3.} We now prove that $E$ satisfies the $rW_\p$-condition for some $r>0$.
Since $E$ is convex, it is enough to show that
\begin{equation}\label{divtau}
{\div}_\tau n\le \frac{1}{r} \qquad \textrm{$\HH^1$-a.e. on $\partial E$.}
\end{equation}
This follows by a local variation argument as in the proof of Lemma~\ref{lm_Euler_mult} below.
Let us fix $x_1\in \Sigma$, where $\Sigma$ is a connected component of $\partial E\cap \Om$,
and $x_2\in\partial E\setminus \bar \Sigma$. We know from the previous step that $\Sigma$
is contained in $\partial W_{r}(x)$ for some $x\in\R^2$ and $r>0$. We distinguish four cases.

\noindent {\sc Case 1.}
There are two disjoint open sets $U_i$, $i=1,2$, such that
$x_i\in U_i$ and $U_i\cap \partial E$ do not contain segments.
Let $\psi_1,\psi_2$ be two nonnegative smooth functions, with support on $U_1,U_2$ respectively,
such that
\begin{equation}\label{condpsi}
\int_{U_1\cap \partial E}\psi_1(z)\varphi^* (\nu(z))\,d\HH^1(z)
=\int_{U_2\cap \partial E}\psi_2(z)\varphi^* (\nu(z))\,d\HH^1(z),
\end{equation}
where $\nu$ stands for the exterior Euclidean unit normal to $\partial E$.
We consider a family of diffeomorphisms such that
$$
\Psi(\varepsilon,x):= x+ \varepsilon \psi_1(x) n(x)-\varepsilon \psi_2(x) n(x)  +o(\eps)
$$
for $\varepsilon>0$ small enough. By \eqref{condpsi}, the term $o(\varepsilon)$ can be chosen in such a way that
\begin{equation}\label{eqess}
|E^\varepsilon| = |E| \qquad \text{for all $\varepsilon >0$ small enough},
\end{equation}
with $E^\varepsilon:= \Psi(\varepsilon,E)\subset\Om$.
We then have
\begin{align*}
P_\p(E^\eps) &= P_\p(E) +  \frac{\varepsilon}{r} \int_{U_1\cap\partial E}\psi_1(z)\varphi^* (\nu(z))\,d\HH^1(z)\\
 & \ \ -\varepsilon
\int_{U_2\cap\partial E}\psi_{2}(z){\div}_\tau n(z)\varphi^* (\nu(z))\,d\HH^1(z) + o(\varepsilon).
\end{align*}
As $\varepsilon \to 0^+$,
by minimality of $E$, we get
\[
\frac{1}{r}
\int_{U_1\cap\partial E}\psi_1(z)\varphi^* (\nu(z))\,d\HH^1(z)
\ge
\int_{U_2\cap\partial E}\psi_{2}(z){\div}_\tau n(z)\varphi^* (\nu(z))\,d\HH^1(z)
\]
which in view of~\eqref{condpsi} gives~\eqref{divtau}.

\noindent {\sc Case 2}.
We can find two maximal segments $\ell_1,\ell_2\subset\partial E$ such that
$x_i\in\ell_i$,
and we define $E^\varepsilon$ by shifting $\ell_1$ by $c_1\varepsilon$ parallel to itself outside
$E$, and by shifting
$\ell_2$ by $c_2\varepsilon$ inside of $E$,
with $c_1,\,c_2$ so that~\eqref{eqess} holds, that is
\begin{equation}\label{eq_Euler3}
c_1|\ell_1|=c_{2}|\ell_{2}|.
\end{equation}
By~\cite[Lemma 4.4]{NovPao05} we have
\begin{align*}
P_\p(E^\eps)   = P_\p(E)
+  c_1 \alpha_1\varepsilon   - c_2\alpha_{2} \varepsilon  + o(\varepsilon).
\end{align*}
where $\alpha_1,\,\alpha_{2}$ are respectively the (Euclidean) length of the face of $W_\varphi$ parallel to $\ell_1,\,\ell_{2}$.
By minimality of $E$, letting $\eps\to 0^+$ we obtain $c_1\alpha_1 \geq c_{2}\alpha_{2}$.
Recalling~\eqref{eq_Euler3}, we finally get
$$
\frac 1 r = \frac{\alpha_1}{|\ell_1|}\ge \frac{\alpha_2}{|\ell_2|}={\div}_\tau n(z)
\qquad \textrm{for}\ z\in\ell_2.
$$

\noindent {\sc Case 3}. There is a maximal segment $\ell_1\subset \partial E$
and an open set $U_2$ such that $x_1\in\ell_1$, $x_2\in U_2$ and
$U_2\cap \partial E$ does not contain segments.
We proceed by combining the previous strategies and we define the set $E^\varepsilon$ by shifting
$\ell_1$ by $\varepsilon$ parallel to itself outside $E$, and then taking the image of the resulting set
through the diffeomorphism
\[
\Psi(\varepsilon,x):= x -\varepsilon \psi_2(x) n(x)  +o(\eps),
\]
where $\psi_2$ is a nonnegative smooth function supported on $U_2$ satisfying
\begin{equation}\label{elleuno}
\int_{U_2\cap \partial E}\psi_2(z)\varphi^* (\nu(z))\,d\HH^1(z) = |\ell_1|.
\end{equation}
This condition guarantees that the volume change after these two operations is of order $o(\eps)$,
so that the extra term $o(\eps)$ in the definition of $\Psi$ is chosen in such a way that~\eqref{eqess} holds.
Reasoning as above, we get
\[
P_\p(E^\eps) = P_\p(E) +\alpha_1\varepsilon -\varepsilon
\int_{U_2\cap\partial E}\psi_{2}(z){\div}_\tau n(z)\varphi^* (\nu(z))\,d\HH^1(z) + o(\varepsilon),
\]
which gives, by minimality of $E$,
\[
\alpha_1=\frac{|\ell_1|}{r}\ge \int_{U_2\cap\partial E}\psi_{2}(z){\div}_\tau n(z)\varphi^* (\nu(z))\,d\HH^1(z)
\]
which gives \eqref{divtau}, recalling \eqref{elleuno}.

\noindent {\sc Case 4}. There is a maximal segment $\ell_2\subset \partial E$
and an open set $U_1$ such that $x_1\in U_1$, $x_2\in \ell_2$ and
$U_1\cap \partial E$ does not contain segments.
This case can be dealt with reasoning as in the previous case, by shifting
$\ell_2$ by $\varepsilon$ inside $E$ and defining
\[
\Psi(\varepsilon,x):= x +\varepsilon \psi_1(x) n_2(x)  +o(\eps).
\]

{\sc Step 4.} From \eqref{divtau} it follows that the
radius $r$ in Step 3 does not depend on the connected component $\Sigma$. In particular,
every connected component
of $\partial E\cap \Om$ is contained in $\partial W_r(x)$, for a fixed $r>0$
(while $x$ depends in general on the connected component).
\end{proof}

Consider now the function $v(r):=|\Om^r|$. It is clearly
constantly equal to $|\Om|$ for $r\le r_\Om$ and to zero for $r>R_\Om$, while over $[r_\Om, R_\Om]$
it is continuous and monotone decreasing.
In particular, for all $m\in [|\Om^{R_\Om}|,|\Om|]$ there exists a unique value $r_m\in [r_\Om,R_\Om]$ such that $v(r_m)=m$.

From the isoperimetric inequality~\eqref{isop} and Lemma~\ref{lm_isoper1}, we get the following statement.

\begin{proposition}\label{proball}
Let $\Om\subset \R^2$ be convex, and let $E$ be a minimizer of~\eqref{isopb} with $m\in [0,|\Om|]$.
Then
\begin{itemize}
\item[(a)] either $\bar E=\bar \Om^{r_m}$, if $m> |\Om^{R_\Om}|$,
\item[(b)] or $\bar E$ is the closure of some
convex union of Wulff balls of radius $R_\Om$, if $m\in [R_\Om^2|W_\p|, |\Om^{R_\Om}|]$,
\item[(c)] or $\bar E=\bar W_{\sqrt{m/|W_\p|}}(x)$ for some $x\in\Om$, if $m\le R_\Om^2|W_\p|$.
\end{itemize}
\end{proposition}

\begin{proof}
We can assume $m\in (0,|\Om|)$.
By Lemma~\ref{lm_isoper1}, there exists an $r>0$ (depending on $m$) such that
$\bar E$ is the closure of a union of Wulff balls of radius $r$, hence
$\bar E\subset \bar \Om^{r}$ and $r\le R_\Om$.

If $m>|\Om^{R_\Om}|$, then necessarily $r<R_\Om$ and
$\bar E = \bar \Om^{r}$, since otherwise we could find a connected component of $\partial E\cap \Om$
which is not contained in the boundary of a Wulff ball, contradicting Lemma~\ref{lm_isoper1}.
In particular, we have $r=r_m$.

If $m\in [R_\Om^2|W_\p|, |\Om^{R_\Om}|]$ then $r= R_\Om$, since otherwise
$\bar E$ would coincide with  the set $\bar \Om^r$ (with $r<R_\Om$) which has volume strictly greater than $|\Om^{R_\Om}|$.

If $m\le R_\Om^2|W_\p|$ the result follows by the isoperimetric inequality~\eqref{isop}.
\end{proof}

\begin{remark}\label{rem_strconv1}
It is worth noticing that, if $\Om$ is strictly convex, then there exists a \emph{unique} Wulff ball
$W_{R_\Om}(x)\subset \Om$, and thus
$\Om^{R_\Om}= W_{R_\Om}(x)$. 
In other words, the case~(b) of the above Proposition~\ref{proball} reduces to case~(c).
Therefore, either $\bar E=\bar \Om^{r_m}$, if $m\ge |\Om^{R_\Om}|$, or
$\bar E=\bar W_{\sqrt{m/|W_{\p}|}}(x)$ for some $x\in\Om$, if $m\le |\Om^{R_\Om}|$.
\end{remark}

We now state an easy consequence of Proposition~\ref{proball}
showing that solutions to problems~\eqref{isopb} with decreasing volumes
may be arranged as a decreasing sequence of sets.

\begin{corollary}\label{corball}
Let $\Om$ be convex and let $m_j$ be a decreasing sequence such that
$m_j\in (0,|\Om|)$, for all $j$.
Then, there exists a sequence of sets $E_j$ such that
$E_{j+1}\subset E_j\subset\Om$, $|E_j|=m_j$ and each
$E_j$ is a minimizer of~\eqref{isopb} with $m:=m_j$.
\end{corollary}

Note that the convexity assumption of the set $\Om$ is essential in the above result. In fact,
reasoning as in~\cite[section~6]{KawLach06} with the example of $\Omega$ a couple of circles connected
by a thin tube (like a barbell considered in~\cite[section~6]{KawLach06}), one
provides
a family of  minimizers of~\eqref{isopb} with decreasing volumes which cannot be arranged as a decreasing sequence of sets
(see Figure~\ref{fig_Kawohl1} below).

\begin{center}
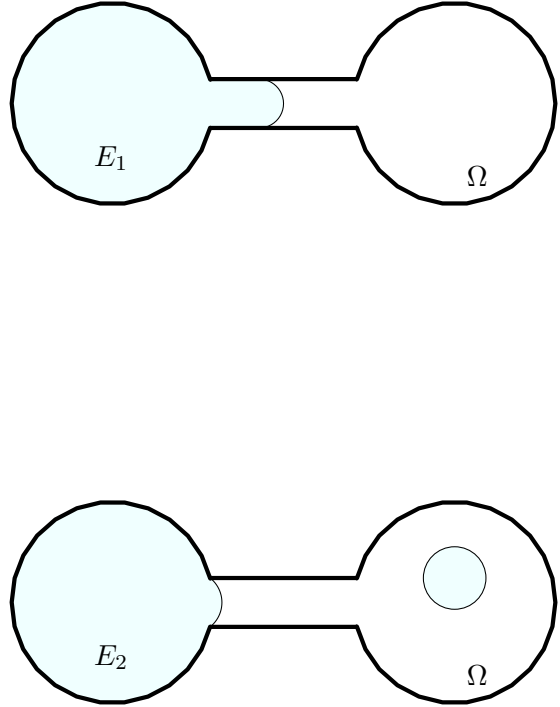
\begin{figure}[ht]\label{fig_Kawohl1}
\definecolor{ccffcc}{rgb}{0.8,1,0.8}
\definecolor{wwffff}{rgb}{0.4,1,1}
\definecolor{xdxdff}{rgb}{0.49,0.49,1}
\definecolor{zzffzz}{rgb}{0.6,1,0.6}
\definecolor{cqcqcq}{rgb}{0.75,0.75,0.75}
\begin{tikzpicture}[line cap=round,line join=round,>=triangle 45,x=0.65cm,y=0.65cm]
\clip(-4.3,-3.86) rectangle (8.26,6.3);
\draw[fill=wwffff,fill opacity=0.1]
plot[domain=0.24:6.04,variable=\t]({-2+ 1*2.06*cos(\t r)+0*2.06*sin(\t r)},{1.5+ 0*2.06*cos(\t r)+1*2.06*sin(\t r)}) -- (0,1)--
plot[domain=-1.57:1.57,variable=\t]({1+1*0.5*cos(\t r)+0*0.5*sin(\t r)},{1.5+ 0*0.5*cos(\t r)+1*0.5*sin(\t r)})
-- (0,2)
-- cycle;
\draw [line width=1.6pt] plot[domain=0.24:6.04,variable=\t]({-2+ 1*2.06*cos(\t r)+0*2.06*sin(\t r)},{1.5+ 0*2.06*cos(\t r)+1*2.06*sin(\t r)}) -- (0,1)-- (3,1) --
plot[domain=-2.9:2.9,variable=\t]({5+ 1*2.06*cos(\t r)+0*2.06*sin(\t r)},{1.5+0*2.06*cos(\t r)+1*2.06*sin(\t r)})
-- (3,2) -- (0,2)
-- cycle;
\draw (5.04,0.42) node[anchor=north west] {$\Omega$};
\draw (-2.58,0.84) node[anchor=north west] {$E_1$};
\end{tikzpicture}
\begin{tikzpicture}[line cap=round,line join=round,>=triangle 45,x=0.65cm,y=0.65cm]
\clip(-4.3,-3.86) rectangle (8.26,6.3);
\draw[fill=wwffff,fill opacity=0.1]
plot[domain=0.24:6.04,variable=\t]({-2+ 1*2.06*cos(\t r)+0*2.06*sin(\t r)},{1.5+ 0*2.06*cos(\t r)+1*2.06*sin(\t r)}) --
plot[domain=-0.9:0.9,variable=\t]({-0.4+1*0.64*cos(\t r)+0*0.64*sin(\t r)},{1.5+0*0.64*cos(\t r)+1*0.64*sin(\t r)})
-- cycle;
\draw[fill=wwffff,fill opacity=0.1] (5,2) circle (0.64);
\draw [line width=1.6pt] plot[domain=0.24:6.04,variable=\t]({-2+ 1*2.06*cos(\t r)+0*2.06*sin(\t r)},{1.5+ 0*2.06*cos(\t r)+1*2.06*sin(\t r)}) -- (0,1)-- (3,1) --
plot[domain=-2.9:2.9,variable=\t]({5+ 1*2.06*cos(\t r)+0*2.06*sin(\t r)},{1.5+0*2.06*cos(\t r)+1*2.06*sin(\t r)})
-- (3,2) -- (0,2)
-- cycle;
\draw (5.04,0.42) node[anchor=north west] {$\Omega$};
\draw (-2.58,0.84) node[anchor=north west] {$E_2$};
\end{tikzpicture}
\caption{$\Omega\subset \R^2$ nonconvex (two circles connected with a thin tube) and two minimizers of~\eqref{isopb}
which cannot be included one into another.}
\end{figure}
\end{center}

\subsection{Isoperimetric towers}\label{sec:isoperimetric-towers}

We return now to the original problem~\eqref{varpb}.
Here and below we let $u\in L^1(\R^2)$ be an arbitrary minimizer of this problem and
$E_j$ be its level set corresponding to a $j\in \N$, as defined by~\eqref{defEj}.
The following result follows directly from Corollary~\ref{corball}.

\begin{proposition}\label{procon}
If $\Om$ is convex, then for all $j\in\N$ the set $E_j$ is a minimizer of problem~\eqref{isopb} with $m:=|E_j|$
(in particular $E_j$ is convex).
\end{proposition}

\begin{proof}
If the assertion is not true, then
considering a sequence of sets $E_j'$
of minimizers of~\eqref{isopb} (with $m:=|E_j|$) such that
$E_{j+1}'\subset E_j'\subset\Om$, $|E_j'|:=|E_j|$
(the existence of such a sequence is guaranteed by Corollary~\ref{corball}), and setting
\[
u':=\sum_j \1_{E_j'},
\]
we get
\[
\int_{\R^2} \p^*(Du')=
\sum_j P_{\p}(E_j') < \sum_j P_{\p}(E_j) = 
\int_{\R^2} \p^*(Du),
\]
the strict inequality being due to the fact that one of $E_j$ is not a minimizer of~\eqref{isopb} (with $m:=|E_j|$) by assumption.
On the other hand,
\[
\int_\Om u'\, dx = \int_\Om u\, dx=m,
\]
since the level sets of $u'$ and $u$ have the same volume by construction.
This would mean that $u$ is not a solution to problem~\eqref{varpb}.
\end{proof}

\begin{remark}\label{remrem}
Observe that, by Proposition~\ref{procon} and Lemma~\ref{lm_isoper1},
each set $E_i$ is convex and each connected component of
$\partial E_i\cap \Omega$ is contained 
in
$\partial W_{r_i}(x_i)$ for some Wulff ball $W_{r_i}(x_i)\subset \Om$. 
\end{remark}

\begin{lemma}\label{lm_Euler_mult}
Let $S_i$, $S_j$ be connected components of $\partial E_i \cap \Omega$ and
$\partial E_j \cap \Omega$ respectively, with $j>i$, such that
\begin{align} \label{sisj}
\nonumber
& S_i  \subset \partial W_{r_i}(x_i)\subset \bar\Omega \\
\nonumber
& S_j \subset \partial W_{r_j}(x_j)\subset\bar\Omega \\
& \frac{1}{r_i} (S_i - x_i) \subset \frac{1}{r_j} (S_j - x_j)
\end{align}
for some $x_i, x_j\in \Omega$, $r_i, r_j >0$.
Then $r_i\geq r_j$.
\end{lemma}

\begin{proof}
It is enough to consider the case $j=i+1$.
We can also assume $S_i\neq S_{i+1}$, otherwise there is nothing to prove.
As in Figure~\ref{twocases}, there are two cases to consider.

\noindent {\sc Case 1.}
There are two points $y_i\in S_i$,
$y_{i+1}\in  S_{i+1}$ and two open sets $U_i
\subset\Omega$ and
$U_{i+1}
\subset\Omega$
such that $y_i\in U_i$, $y_{i+1}\in U_{i+1}$,
$U_i\cap U_{i+1}=\emptyset$ and $U_i\cap S_i$ as well as $U_{i+1}\cap S_{i+1}$
does not contain segments. Consider a smooth function
$\psi_i$ with support on $U_i$.
It generates a one-parameter family of diffeomorphisms of $E_i$ defined by
\[
\Psi_i(\varepsilon,x):= x- \varepsilon \psi_i(x) n_i(x)
\]
for all sufficiently small $\varepsilon>0$, where
\[
n_i(x):=\frac{x-x_i}{r_i \varphi(x-x_i)}.
\]
Consider now a one-parameter family $\{\Psi_{i+1}(\varepsilon,\cdot)\}$ of diffeomorphisms of
$E_{i+1}$ such that $\Psi_{i+1}(0,x)=x$ for all $x\in E_{i+1}$,
$\Psi_{i+1}(\varepsilon,\cdot)-\mbox{Id}$
is supported in $U_{i+1}$
for all $\varepsilon >0$, while
\[
\Psi_{i+1}(\varepsilon,x):=x+\varepsilon \psi_{i+1}(x)n_{i+1}(x)+o(\varepsilon)
\]
as $\varepsilon\to 0^+$, where $\psi_{i+1}$ is some smooth function
(with support in $U_{i+1}$),
and
\[
n_{i+1}(x):=\frac{x-x_{i+1}}{r_{i+1} \varphi(x-x_{i+1})}.
\]
We choose $\Psi_{i+1}$ so that the sets
$E_i^\varepsilon:= \Psi_i(\varepsilon,E_i)$ and
$E_{i+1}^\varepsilon:= \Psi_{i+1}(\varepsilon,E_{i+1})$ satisfy
\[
|E_i^\varepsilon| + |E_{i+1}^\varepsilon| = |E_i| + |E_{i+1}|
\]
for all sufficiently small $\varepsilon >0$.
Denote by $\nu_j$ the exterior Euclidean unit normal to $\partial E_j$.
Since
\begin{align*}
|E_i^\varepsilon| &= |E_i| - \varepsilon
\int_{\partial E_i\cap U_i}\psi_i(z)\varphi^* (\nu_i(z))\,d\HH^1(z) + o(\varepsilon),\\
|E_{i+1}^\varepsilon| &= |E_{i+1}| + \varepsilon \int_{\partial E_{i+1}\cap U_{i+1}}\psi_{i+1}(z)\varphi^* (\nu_{i+1}(z))\,d\HH^1(z) + o(\varepsilon),
\end{align*}
as $\varepsilon \to 0^+$,
we have
\begin{equation}\label{eq_Euler1aux}
\begin{array}{l}
\displaystyle
\int_{\partial E_i\cap U_i}\psi_i(z)\varphi^* (\nu_i(z))\,d\HH^1(z)=\\
\displaystyle\qquad\qquad\int_{\partial E_{i+1}\cap U_{i+1}}\psi_{i+1}(z)\varphi^* (\nu_{i+1}(z))\,d\HH^1(z).
\end{array}
\end{equation}
Letting now
\[
u_\varepsilon := u - \1_{E_i}-\1_{E_{i+1}} + \1_{E_i^\varepsilon}+\1_{E_{i+1}^\varepsilon}
= \sum_{k\neq i, k\neq i+1} \1_{E_k} + \1_{E_i^\varepsilon}+\1_{E_{i+1}^\varepsilon},
\]
we have
$\int_\Omega u_\varepsilon\,dx = \int_\Omega u\,dx$ for all
sufficiently small $\varepsilon >0$.
Recall that
\begin{align*}
\int_\Omega \varphi^*(D u_\varepsilon)  & =
\int_\Omega \varphi^*(D u) \\
 &  -  \varepsilon \int_{\partial E_i \cap U_i}\frac{1}{r_i}\psi_i(z)\varphi^* (\nu_i(z))\,d\HH^1(z)\\
 & +\varepsilon
\int_{\partial E_{i+1}\cap U_{i+1}}\frac{1}{r_{i+1}}\psi_{i+1}(z)\varphi^* (\nu_{i+1}(z))\,d\HH^1(z) \\
& + o(\varepsilon).
\end{align*}
As $\varepsilon \to 0^+$,
by minimality of $u$, we get
\[
\begin{array}{l}
-\displaystyle\frac{1}{r_i}
\int_{\partial E_i \cap U_i}\psi_i(z)\varphi^* (\nu_i(z))\,d\HH^1(z)\\
\displaystyle\qquad\qquad+
\frac{1}{r_{i+1}}\int_{\partial E_{i+1}\cap U_{i+1}}\psi_{i+1}(z)\varphi^* (\nu_{i+1}(z))\,d\HH^1(z)\geq 0,
\end{array}
\]
which together with~\eqref{eq_Euler1aux} implies the thesis.

\noindent {\sc Case 2}.
We can find two maximal line segments $\ell_i\subset S_i$ and $\ell_{i+1}\subset S_{i+1}$.
We define then
$E_i^\varepsilon$ by shifting the segment $\ell_i$ by $c_i\varepsilon$ parallel to itself inside
$E_i$ and
$E_{i+1}^\varepsilon$ by
shifting the segment $\ell_{i+1}$  parallel to itself outside of $E_{i+1}$ by
$c_{i+1}\varepsilon$  with $c_i$ and $c_{i+1}$ so as to satisfy \[
|E_i^\varepsilon| + |E_{i+1}^\varepsilon| = |E_i| + |E_{i+1}|
\]
for all $\varepsilon >0$ sufficiently small.
Since
\begin{align*}
|E_i^\varepsilon| &= |E_i| - c_i |\ell_i|\varepsilon
+ o(\varepsilon),\\
|E_{i+1}^\varepsilon| &= |E_{i+1}| + c_{i+1} |\ell_{i+1}|\varepsilon + o(\varepsilon),
\end{align*}
as $\varepsilon \to 0^+$,
we have
\begin{equation}\label{eq_Euler2aux}
c_i|\ell_i|=c_{i+1}|\ell_{i+1}|.
\end{equation}
Letting again, as in Case~1,
\[
u_\varepsilon := u - \1_{E_i}-\1_{E_{i+1}} + \1_{E_i^\varepsilon}+\1_{E_{i+1}^\varepsilon}
= \sum_{k\neq i, k\neq i+1} \1_{E_k} + \1_{E_i^\varepsilon}+\1_{E_{i+1}^\varepsilon},
\]
we have
$\int_\Omega u_\varepsilon\,dx = \int_\Omega u\,dx$ for all
sufficiently small $\varepsilon >0$.
On the other hand, by~\cite[lemma 4.4]{NovPao05}
\begin{align*}
\int_\Omega \varphi^*(D u_\varepsilon)   =
\int_\Omega \varphi^*(D u) -  c_i \alpha_i\varepsilon   +c_{i+1}\alpha_{i+1} \varepsilon  + o(\varepsilon).
\end{align*}
where $\alpha_i,\,\alpha_{i+1}$ are respectively the (Euclidean) length of the face of $W_\varphi$ parallel to $\ell_i,\,\ell_{i+1}$.
By minimality of $u$, letting $\eps\to 0^+$ we obtain $c_i\alpha_i \leq c_{i+1}\alpha_{i+1}$.
Recalling~\eqref{eq_Euler2aux}, we get
$r_i=|\ell_i|/\alpha_i\geq |\ell_{i+1}|/\alpha_{i+1}=r_{i+1}$.

Notice that in this proof we do not have to deal with the situation depicted in Cases~3 and~4 of the proof of Lemma~\ref{lm_isoper1} due to condition~\eqref{sisj}.
In fact, the latter implies that if $S_i$ contains a line segment $\ell_i$, then the line segment $\ell_j:=x_j +  (l_i-x_i)r_j/r_i$ is contained in $S_j$. Otherwise, if there is a neighborhood $U_i$ of a point of $S_i$ such that $S_i\cap U_i$ does not contain any line segment, then $U_j:=x_j + (U_i-x_i)r_j/r_i$ is a neighborhood of a point in $S_j$ such that
$S_j\cap U_j$ does not contain any line segment.
\end{proof}

\begin{center}
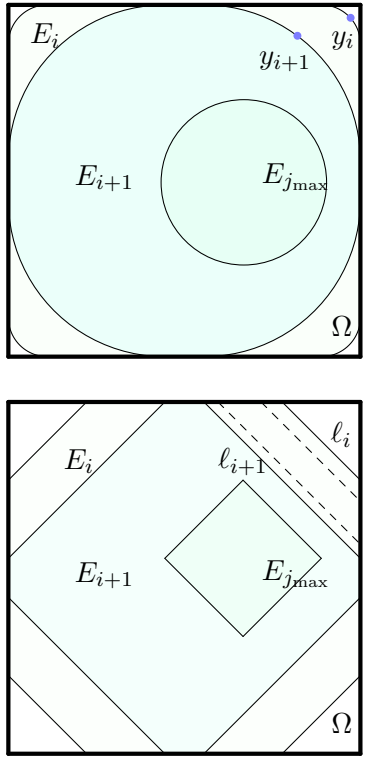
\begin{figure}[ht]\label{twocases}
\definecolor{ccffcc}{rgb}{0.8,1,0.8}
\definecolor{wwffff}{rgb}{0.4,1,1}
\definecolor{xdxdff}{rgb}{0.49,0.49,1}
\definecolor{zzffzz}{rgb}{0.6,1,0.6}
\definecolor{cqcqcq}{rgb}{0.75,0.75,0.75}
\begin{tikzpicture}[line cap=round,line join=round,>=triangle 45,x=0.52cm,y=0.52cm]
\clip(-4.3,-3.8) rectangle (9.58,6.3);
\draw [
fill=ccffcc,fill opacity=0.1]
plot[domain=0:1.57,variable=\t]({4+1*1*cos(\t r)+0*1*sin(\t r)},{5+0*1*cos(\t r)+1*1*sin(\t r)}) -- (-3,6) --
plot[domain=1.57:pi,variable=\t]({-3+1*1*cos(\t r)+0*1*sin(\t r)},{5+ 0*1*cos(\t r)+1*1*sin(\t r)}) -- (-4,-2) --
plot[domain=pi:4.71,variable=\t]({-3+1*1*cos(\t r)+0*1*sin(\t r)},{-2+ 0*1*cos(\t r)+1*1*sin(\t r)}) -- (4,-3) --
plot[domain=-1.57:0,variable=\t]({4+ 1*1*cos(\t r)+0*1*sin(\t r)},{-2+0*1*cos(\t r)+1*1*sin(\t r)})
 -- cycle ;
\draw [
fill=wwffff,fill opacity=0.1]
plot[domain=0:1.57,variable=\t]({1+1*4*cos(\t r)+0*4*sin(\t r)},{2+ 0*4*cos(\t r)+1*4*sin(\t r)}) -- (0,6) --
plot[domain=1.57:pi,variable=\t]({1*4*cos(\t r)+0*4*sin(\t r)},{2+0*4*cos(\t r)+1*4*sin(\t r)}) -- (-4,1)
   --  plot[domain=pi:4.71,variable=\t]({1*4*cos(\t r)+0*4*sin(\t r)},{1+ 0*4*cos(\t r)+1*4*sin(\t r)}) -- (1,-3)
 --  plot[domain=-1.57:0,variable=\t]({1+ 1*4*cos(\t r)+0*4*sin(\t r)},{1+ 0*4*cos(\t r)+1*4*sin(\t r)})
 -- cycle ;
\draw [line width=1.6pt] (-4,-3)-- (5,-3);
\draw [line width=1.6pt] (5,-3)-- (5,6);
\draw [line width=1.6pt] (5,6)-- (-4,6);
\draw [line width=1.6pt] (-4,6)-- (-4,-3);
\draw [
fill=zzffzz,fill opacity=0.05] (2.02,1.46) circle (1.1cm);
\draw (4.0,-1.7) node[anchor=north west] {\parbox{2.2 cm}{$ \Omega \\  $}};
\draw (-3.7,5.79) node[anchor=north west] {$E_i$};
\draw (-2.56,2.16) node[anchor=north west] {$E_{i+1}$};
\draw (2.22,2.22) node[anchor=north west] {$E_{j_{\rm max}}$};
\draw (2.15,5.06) node[anchor=north west] {$y_{i+1}$};
\draw (4.00,5.58) node[anchor=north west] {$y_{i}$};
\fill [color=xdxdff] (3.39,5.21) circle (1.5pt);
\fill [color=xdxdff] (4.75,5.67) circle (1.5pt);
\end{tikzpicture}
\begin{tikzpicture}[line cap=round,line join=round,>=triangle 45,x=0.52cm,y=0.52cm]
\clip(-4.3,-3.8) rectangle (9.58,6.3);
\fill[line width=1.2pt,color=ccffcc,fill=ccffcc,fill opacity=0.05] (-4,4) -- (-2,6) -- (3,6) -- (5,4) -- (5,-1) -- (3,-3) -- (-2,-3) -- (-4,-1) -- cycle;
\fill[line width=1.2pt,color=wwffff,fill=wwffff,fill opacity=0.05] (-4,2) -- (0,6) -- (1,6) -- (5,2) -- (5,1) -- (1,-3) -- (0,-3) -- (-4,1) -- cycle;
\fill[line width=1.2pt,color=zzffzz,fill=zzffzz,fill opacity=0.05] (2,0) -- (4,2) -- (2,4) -- (0,2) -- cycle;
\draw [line width=1.6pt] (-4,-3)-- (5,-3);
\draw [line width=1.6pt] (5,-3)-- (5,6);
\draw [line width=1.6pt] (5,6)-- (-4,6);
\draw [line width=1.6pt] (-4,6)-- (-4,-3);
\draw  (-4,2)-- (0,6);
\draw  (0,6)-- (1,6);
\draw  (1,6)-- (5,2);
\draw  (5,2)-- (5,1);
\draw  (5,1)-- (1,-3);
\draw  (1,-3)-- (0,-3);
\draw  (0,-3)-- (-4,1);
\draw  (-4,1)-- (-4,2);
\draw  (-4,4)-- (-2,6);
\draw  (-2,6)-- (3,6);
\draw  (3,6)-- (5,4);
\draw  (5,4)-- (5,-1);
\draw  (5,-1)-- (3,-3);
\draw  (3,-3)-- (-2,-3);
\draw  (-2,-3)-- (-4,-1);
\draw  (-4,-1)-- (-4,4);
\draw  (2,0)-- (4,2);
\draw  (4,2)-- (2,4);
\draw  (2,4)-- (0,2);
\draw  (0,2)-- (2,0);
\draw [dash pattern=on 3pt off 3pt] (2.46,6)-- (5,3.46);
\draw [dash pattern=on 3pt off 3pt] (1.36,6)-- (5,2.36);
\draw (4.0,-1.7) node[anchor=north west] {\parbox{2.2 cm}{$ \Omega \\  $}};
\draw (-2.84,5.06) node[anchor=north west] {$E_i$};
\draw (-2.56,2.16) node[anchor=north west] {$E_{i+1}$};
\draw (2.22,2.22) node[anchor=north west] {$E_{j_{\rm max}}$};
\draw (1.1,5.08) node[anchor=north west] {$\ell_{i+1}$};
\draw (3.98,5.74) node[anchor=north west] {$\ell_i$};
\end{tikzpicture}
\caption{The two possible cases in the proof of Lemma~\ref{lm_Euler_mult}.}
\end{figure}
\end{center}

We are now able to prove the following result giving the complete characterization of solutions to problem~\eqref{varpb}.

\begin{theorem}\label{th_mcheeg1final}
Let $\Om\subset \R^2$ be convex
and set $\j:=\|u\|_\infty$.
Then one of the following cases holds.
\begin{enumerate}
\item[a)] There exists an $\bar r\in [r_\Om,R_\Om)$ such that $\bar E_j=\bar \Om^{\bar r}$ for all $j\le\j$. In this case
    \[
    u= \j \1_{\Om^{\bar r}}
    \]
    (in particular, if $\bar r=r_\Om$,
    then $u= \j \1_{\Om}$).
\item[b)] There exists an $\bar r\in (r_\Om,R_\Om)$ such that $\bar E_\j= \bar W_{\bar r}(x)$
for some $x\in \Omega$ with $W_{\bar r}(x)\subset \Om^{\bar r}$,
and $\bar E_j=\bar \Om^{\bar r}$ for all $j<\j$.
In this case
    \[
    u= \1_{W_{\bar r}(x)} + (\j-1) \1_{\Om^{\bar r}}.
    \]
\item[c)] There exists an $\bar r\in (0,r_\Om]$ such that
$\bar E_\j= \bar W_{\bar r}(x)$
for some $x\in \Omega$ with $W_{\bar r}(x)\subset \Om$,
and
$\bar E_j=\bar \Om$ for all $j<\j$.
In this case
    \[
    u= \1_{W_{\bar r}(x)} + (\j-1) \1_{\Om}
    \]
    (note that this condition may hold only when $r_\Om>0$).
\item[d)] Every $\bar E_j$ is the closure of a convex union of Wulff balls of radius $R_\Om$ for all $j\le\j$.
\end{enumerate}
\end{theorem}

\begin{remark}
Observe that case~d) of Theorem~\ref{th_mcheeg1final} is the only case
where the number of nonzero level sets of the minimizer may be bigger than two.
\end{remark}

\begin{proof}
We may assume $\j>1$, since otherwise the result follows directly from
Proposition~\ref{proball}.

By Remark~\ref{remrem}, for all $i\le \j$ the set $E_i$ is convex
and each connected component of
$\partial E_i\cap \Omega$ is contained, up to a translation, in
$\partial W_{r_i}(x_i)$ for some $r_i>0$, $x_i\in \Omega$. Moreover,
if $\partial E_i\cap \Omega$ and $\partial E_{i+1}\cap \Omega$
are nonempty, from the inclusion $E_{i+1}\subset E_i$ it follows that
we can always find two connected components $S_i\subset \partial E_i\cap \Omega$
and $S_{i+1}\subset \partial E_{i+1}\cap \Omega$ satisfying the assumptions of Lemma~\ref{lm_Euler_mult}.
By Lemma~\ref{lm_Euler_mult} we then get $r_i\geq r_{i+1}$ for all $i<\j$.

Recalling Propositions~\ref{procon} and \ref{proball}, this leaves only the following possibilities:
\begin{itemize}
  \item[(i)] $\bar E_i=\bar \Om^{r_i}$ and $\bar E_{i+1}=\bar \Om^{r_{i+1}}$,
          with $r_{i}\geq r_{i+1}$. In this case, if $r_i>r_\Om$ so that $\bar \Om^{r_i} \neq\bar \Om$, then
          $r_{i}= r_{i+1}$, hence $\bar E_i=\bar E_{i+1}=\bar \Om^{r_i}$,
          while if $r_{i+1}\leq r_i\leq r_\Om$ we have $\bar E_i=\bar E_{i+1}=\bar \Om$, and we may just set
          $r_{i}= r_{i+1}:=r_\Om$ so that we still have $\bar E_i=\bar E_{i+1}=\bar \Om^{r_i}=\bar\Om$.
  \item[(ii)] $\bar E_i=\bar \Om^{r_i}$ and $\bar E_{i+1}=\bar W_{r_{i+1}}(x_{i+1})$, with $r_{i}\geq r_{i+1}$.
  \item[(iii)] $\bar E_i=\bar W_{r_i}(x_i)$ and $\bar E_{i+1}=\bar W_{r_{i+1}}(x_{i+1})$, with $r_{i}\geq r_{i+1}$.
  \item[(iv)] $\bar E_i$ is a closure of a convex union of Wulff shapes
  of radius $R_\Om$ and $E_{i+1}=\bar W_{r_{i+1}}(x_{i+1})$, with $R_\Om> r_{i+1}$
  (note that the case where $\bar E_{i+1}$ is the closure of a convex union of Wulff balls of radius $R_\Om$ and $\bar E_{i}=\bar W_{r_{i}}(x_{i})$ with $R_\Om< r_{i}$ is impossible).
  \item[(v)] $\bar E_i$ and $\bar  E_{i+1}$ are both the closure of a convex union of Wulff shapes of radius $R_\Om$.
\end{itemize}
Thus there is $\bar\jmath\in \{0,\ldots, \j\}$ and $\bar r\in [r_\Om, R_\Om)$ such that, for every $i\leq\bar\jmath$,
  \begin{itemize}
  \item[(A)]
  either $\bar E_i=\bar \Om^{\bar r}$, and
  $\bar E_i=\bar W_{r_i}(x_{i})$ with $r_i< \bar r$ for all $i>\bar\jmath$;
  \item[(B)]
  or $\bar E_i$ is the closure of a convex union of Wulff shapes of radius $R_\Om$,
  and  $\bar E_i=\bar W_{r_i}(x_{i})$ with $r_i< R_\Om$ for all $i>\bar\jmath$.
  \end{itemize}

We want to show that $\bar\jmath= \j-1$ or $\bar\jmath =\j$.

Assume by contradiction that $\j\ge \bar\jmath+2$.
As $r_{\bar\jmath+1}< \bar r$ in Case~A, and $r_i<R_\Om$ in Case~B,
without loss of generality we can assume that $\bar E_{\bar\jmath+1}= \bar W_{r_{\bar\jmath+1}}(x_{\bar\jmath+1})$
is contained in the interior of $\bar E_{\bar\jmath}$. Then, we can
choose $\varepsilon,\varepsilon'>0$ sufficiently small so that, letting
$r_{\j}':=r_{\j}-\varepsilon'$ and
$r_{\bar\jmath+1}' := r_{\bar\jmath+1}+\varepsilon$, we have $\bar W_{r_{\bar\jmath+1}}(x_{\bar\jmath+1}) \subset \bar E_{\j}$, and
\begin{equation}\label{eq_rjk2}
{r'}_{\j}^{2}+ {r'}_{\bar\jmath+1}^{2} =r_{\j}^2+ r_{\bar\jmath+1}^2.
\end{equation}
{}From~\eqref{eq_rjk2} we then get
\[
\varepsilon'=\frac{r_{\bar\jmath+1}}{r_{\j}}\varepsilon + o(\varepsilon),
\]
and hence
\begin{eqnarray*}
\frac{P_\p (W_{r'_{\bar\jmath+1}}(x_{\bar\jmath+1}))
+ P_\p (W_{r'_{\j}}(x_{\j}))}
{P_\p(W_{r_{\bar\jmath+1}}(x_{\bar\jmath+1}))
+ P_\p (W_{r_{\j}}(x_{\j}))} &=& \frac{r'_{\bar\jmath+1}+r'_{\j}}{r_{\bar\jmath+1}+r_{\j}}
\\
&=&
1-\varepsilon \frac{r_{\bar\jmath+1}-r_{\j}}{r_{\j}(r_{\bar\jmath+1}+r_\j)} +o(\varepsilon),
\end{eqnarray*}
where the error term $o(\varepsilon)$ is negative whenever $r_{\bar\jmath+1}=r_\j$.
Representing $u$ as
$u= \tilde u+ \1_{E_{\bar\jmath+1}}+ \1_{E_{\j}}$,
and letting
\[
u'_\varepsilon:= \tilde u+ \1_{W_{r'_{\bar\jmath+1}}(x_{\bar\jmath+1})}+ \1_{W_{r'_{\j}}(x_{\j})},
\]
we then get
\[
\int_\Omega \p^*(Du'_\varepsilon) < \int_\Omega \p^*(Du),
\]
for sufficiently small $\varepsilon>0$.
Since $\int_{\R^2} u\, dx = \int_{\R^2} u'_\varepsilon\, dx$, this contradicts the minimality of $u$.
One has therefore that $\bar\jmath= \j-1$ or $\bar\jmath =\j$, which concludes the proof.
\end{proof}

\section{An explicit example}\label{secexa}
\subsection{A square with the Euclidean norm}

Let now $\Om:=[0,1]^2$ and let $\p$ be the Euclidean norm on $\R^2$.
From Theorem~\ref{th_mcheeg1final} we obtain the following characterization for the minimizers of \eqref{varpb}.

\begin{proposition}\label{prosq}
Let $\Om$ and $\p$ be as above.
\begin{itemize}
\item[(i)] If $m\in (n-1,n\pi/4)$, with $1\le n\le 4$, we have $\j=n$, $\bar E_\j=\bar B_r(x_0)\subset \bar\Om$
and $\bar E_j=\bar \Om^r$ for $j<\j$, with
\[
r=\sqrt{(n-m-1)/(4(n-1)-n\pi)}
\]
\item[(ii)] If $m\in [n\pi/4,n]$, with $1\le n\le 4$, we have $\j=n$ and $\bar E_j=\bar \Om^r$ for $j\le\j$, with $r=\sqrt{(n-m)/((4-\pi)n)}$.
\item[(iii)] If $m> 4$ we have
\begin{equation}\label{eqAB}
\j\in\left\{
\left\lfloor \frac{2+\sqrt\pi}{2\sqrt\pi}m\right\rfloor ,
\left\lfloor \frac{2+\sqrt\pi}{2\sqrt\pi}m\right\rfloor +1
\right\}
\end{equation}
and $\bar E_j=\bar \Om^r$ for $j\le\j$, with $r=\sqrt{(1-m/\j)/(4-\pi)}$.
\end{itemize}
\end{proposition}

\begin{proof}
Clearly, $r_\Omega=0$, $R_\Omega=1/2$.
By Theorem~\ref{th_mcheeg1final} for all $m>0$ we have one of the following two possibilities.

{\sc Case A.} $\bar E_j=\bar \Om^r$ for all $j\le \j$ with
\[
m=\j|\Om^r|=\j\left( 1-(4-\pi)r^2\right)
\qquad r\in[0,1/2].
\]
It then follows
\[
r=r_A(\j) := \sqrt{\frac{\j-m}{(4-\pi)\j}}
\]
and $\sum_{j=1}^\j P(E_j)=F_A(\j)$, where
\[
F_A(x) := x P(\Om^{r_A(x)}) = 4x-2\sqrt{4-\pi}\,\sqrt{x(x-m)}.
\]
Notice that
\[
F_A'(x)= 4 - \sqrt{4-\pi}\,\frac{2x-m}{\sqrt{x(x-m)}}
\]
which implies that $F_A(x)$ is increasing for $x>\frac{2+\sqrt\pi}{2\sqrt\pi}m$ and decreasing for
$m\le x<\frac{2+\sqrt\pi}{2\sqrt\pi}m$. As a consequence we have
\begin{equation*}
\j\in
\{j^A, j^A+1\},
\mbox{ where }
j^A:=
\left\lfloor \frac{2+\sqrt\pi}{2\sqrt\pi}m\right\rfloor .
\end{equation*}

{\sc Case B.} $\bar E_\j=\bar B_r(x_0)\subset \bar\Om$ and $\bar E_j=\bar \Om^r$ for all $j< \j$ with
\begin{align*}
m &=\pi r^2+(\j-1)|\Om^r|\\
  & =(\j-1)\left( 1-\left(4-\frac{\j}{\j-1}\pi\right)r^2\right)\qquad r\in(0,1/2).
\end{align*}
It follows that 
\[
r=r_B(\j) := \sqrt{\frac{\j-1-m}{(4-\pi)(\j-1)-\pi}}
\]
and $\sum_{j=1}^\j P(E_j)=F_B(\j)$, where
\begin{eqnarray*}
F_B(x) &:=& (x-1) P(\Om^{r_B(x)}) + 2\pi r_B(x)
\\
&=& 4(x-1)-2\sqrt{4-\pi}\,\sqrt{\left(x-\frac{4}{4-\pi}\right)(x-1-m)}.
\end{eqnarray*}
Notice that the derivative
\[
F_B'(x)= 4 - \sqrt{4-\pi}\,\frac{2(x-1)-m-\dfrac{\pi}{4-\pi}}{\sqrt{\left(x-\dfrac{4}{4-\pi}\right)(x-1-m)}}.
\]
Assuming $x\ge \frac{4}{4-\pi}$, we then have
that $F_B$ is increasing for
\[
x>\frac{2+\sqrt\pi}{2\sqrt\pi} m+\frac{4+\sqrt\pi}{2(2+\sqrt\pi)}
\]
and decreasing otherwise, so that
\begin{equation*}
\j\in \{j^B, j^B+1\},
\mbox{ where }
j^B:=
\left\lfloor \frac{2+\sqrt\pi}{2\sqrt\pi} m+\frac{4+\sqrt\pi}{2(2+\sqrt\pi)}\right\rfloor,
\end{equation*}
as soon as $\j\ge 5$.

Observe that, if $\pi \leq m < 5\pi/4$, then 
Case~B cannot occur. In fact, supposing the contrary, we would have
$\j\in \{j^B, j^B+1\}=\{4,5\}$, while
one should have $4- \pi \j/(\j-1) \geq 0$, i.e.\ $\j\geq \lfloor 4/(4-\pi)\rfloor=5$
which implies $\j=5$. On the other hand either
\begin{itemize}
\item[(1)] $\j-1\geq \pi/(4-\pi)$ with $\j\geq m+1$, which would mean $\j\geq 6$, so that this case is impossible, or 
\item[(2)] $\j-1<  \pi/(4-\pi)$ with $\j\leq m+1$, which means $\j=4$.
\end{itemize}
This gives the contradiction proving that Case~B does not occur.

If $m>5\pi/4$, we have to determine which one between Case A and B is energetically more convenient.
By a tedious calculation one proves that
$\min\{F_B(j^B),F_B(j^B+1)\} >\min\{F_A(j^A),F_A(j^A+1)\}$,
which shows that Case B still cannot occur. This shows claim~(iii) of the statement being proven.

Finally, with $m< \pi$ one can only have $\j\in \{1,2,3,4\}$. It is easy to verify then 
that when $m\in (n-1,n\pi/4)$ and $j=n\in \{1,2,3,4\}$, one has $F_B(j)<F_A(j)$, so that Case~B occurs, thus proving claim~(i), while
when $m\in [n\pi/4,n]$ and $j=n\in \{1,2,3,4\}$, one has $F_A(j)\leq F_B(j)$, and hence Case~A occurs, thus proving claim~(ii).
\end{proof}

Notice that $\frac{2\pi}{2+\pi}$ is the volume of the (unique) Cheeger set $C_\Om$ of $\Om$,
so that Proposition~\ref{prosq} implies that the functions $u_m/\j$ converge to the characteristic function of $C_\Om$,
according to the Remark~\ref{remlim}.

\subsection{A square with a crystalline norm}

Now we set $\Om=[0,1]^2$ as above and $\p(\nu)=|\nu_1|+|\nu_2|$. Notice that
$\p$ is a crystalline norm with Wulff shape $W_\p=\{(x, y)\in\R^2:\,|x|+|y|< 1\}$.
As before, we are able to characterize completely the minimizers of~\eqref{varpb}.

\begin{proposition}
Let $\Om$ and $\p$ be as above.
\begin{itemize}
\item[(i)] If $m\in (0,1/2]$, we have $\j=1$ and $\bar E_1=\bar W_r(x_0)\subset \bar\Om$, with $r=\sqrt{m/2}$.
\item[(ii)] If $m\in [1/2,1)$, we have $\j=1$ and $\bar E_1=\bar\Om^r$, with $r=\sqrt{(1-m)/2}$.
\item[(iii)] If $m=1$, then either $\j=1$ and $\bar E_1=\Om$, or $\j=2$, $\bar E_1=\bar \Om^r$ and $\bar E_2=\bar W_r(x_0)\subset\Om$,
with $r\in (0,1/2]$.
\item[(iv)] If $m>1$, we have 
\begin{equation}\label{eqCD}
\j\in\left\{
\left\lfloor \frac{1+\sqrt 2}{2}m\right\rfloor ,
\left\lfloor \frac{1+\sqrt 2}{2}m\right\rfloor +1
\right\}
\end{equation}
and $\bar E_j=\bar \Om^r$ for $j\le\j$, with $r=\sqrt{(1-m/\j)/2}$.
\end{itemize}
\end{proposition}

\begin{proof}
The proof is similar to the one of Proposition~\ref{prosq}.

Clearly, if $m\le 1/2$, then $\j=1$ and $\bar E_1=\bar W_r(x_0)\subset \Om $, since the (rescaled) Wulff shape solves the isoperimetric problem.
By Theorem~\ref{th_mcheeg1final}, for all $m\ge 1/2$ we have one of the following two possibilities.

{\sc Case A.} $\bar E_j=\bar \Om^r$ for all $j\le \j$ with
\[
m=\j|\Om^r|= \j\left( 1-2r^2\right)\qquad r\in[0,1/2],
\]
which gives
\[
r=\frac{1}{\sqrt 2}\sqrt{1-\frac{m}{\j}},
\]
and $\sum_{j=1}^\j P_\p(E_j) = F_A(\j)$, where
\[
F_A(x) = x(4-4r)= 4x-4\sqrt\frac{x^2-mx}{2}\,.
\]
Since the function $F_A$ is increasing for
$x>\frac{1+\sqrt 2}{2}m$ and decreasing for $m\le x<\frac{1+\sqrt 2}{2}m$,
we have
\begin{equation*}
\j\in
\{j^A, j^A+1\},
\mbox{ where }
j^A:=
\left\lfloor \frac{1+\sqrt 2}{2} m\right\rfloor.
\end{equation*}

{\sc Case B.} $\bar E_\j=\bar W_r(x_0)\subset \bar\Om$ and $\bar E_j=\bar\Om^r$ for all $j< \j$,
with $r\in (0,1/2]$ and
\[
m =2r^2+(\j-1)\left( 1-2r^2\right),
\]
and hence $m\geq 1$ because $\j\geq 2$ and $r\leq 1/2$.

If $m=1$ then $\j=2$ and we can take any $r\in (0,1/2]$.

If $m>1$ then $\j\ge m+1$ and we get
\[
r=\sqrt\frac{\j-m-1}{2(\j-2)}
\]
and $\sum_{j=1}^\j P_\p(E_j)=F_B(\j)$, where
\[
F_B(x)=4(x-1)-4(x-2)\sqrt\frac{x-m-1}{2(x-2)}\,.
\]
Since the function $F_B$ is increasing for $x>\frac{1+\sqrt 2}{2}m + \frac{3-\sqrt 2}{2}$ and decreasing otherwise,
we have
\begin{equation*}
\j\in \{j^B, j^B+1\},
\mbox{ where }
j^B:=
\left\lfloor \frac{1+\sqrt 2}{2} m+\frac{3-\sqrt 2}{2}\right\rfloor.
\end{equation*}

As in the proof of Proposition~\ref{prosq}, when $m>1$ we have to determine which one between Cases~A 
and~B is energetically more convenient.
Since 
$\min\{F_B(j^B),F_B(j^B+1)\} >\min\{F_A(j^A),F_A(j^A+1)\}$
(again, by a tedious calculation like in the example with Euclidean norm),
it follows that Case~B can never occur.
\end{proof}


\end{document}